\UseRawInputEncoding

\documentclass[12pt]{amsart}
\usepackage{amsfonts}
\usepackage{amsmath}
\usepackage{amssymb,latexsym}
\usepackage[mathcal]{eucal}
\usepackage[pdftex,bookmarks,colorlinks,breaklinks]{hyperref}

\input xy
\xyoption{all}

\newcommand{\Bcal}{\mathcal{B}}
\newcommand{\Ccal}{\mathcal{C}}
\newcommand{\Dcal}{\mathcal{D}}
\newcommand{\Ecal}{\mathcal{E}}

\newcommand{\Mcal}{\mathcal{M}}

\newcommand{\Tcal}{\mathcal{T}}
\newcommand{\Ucal}{\mathcal{U}}

\newcommand{\Z}{\mathbb{Z}}

\newcommand{\A}{\mathbb{A}}

\newcommand{\Rb}{\mathbf{R}}

\newcommand{\Tb}{\mathbf{T}}
\newcommand{\Sb}{\mathbf{S}}

\newcommand{\Xb}{\mathbf{X}}
\newcommand{\Yb}{\mathbf{Y}}
\newcommand{\Zb}{\mathbf{Z}}

\newcommand{\del}{\delta}

\newcommand{\la}{\lambda}

\newcommand{\ka}{\kappa}

\newcommand{\ol}{\overline}

\newcommand{\br}{\vspace{3 mm}}
\newcommand{\imp}{\Rightarrow}

\newcommand{\rest}{\upharpoonright}

\newcommand{\diam}{{\rm{diam\,}}}

\swapnumbers
\theoremstyle{plain}
\newtheorem{thm}{Theorem}[section]

\newtheorem{lem}[thm]{Lemma}
\newtheorem{prop}[thm]{Proposition}

\theoremstyle{definition}

\newtheorem{question}[thm]{Question}

\begin{document}


\title
[On systems  disjoint from every ergodic system]
{On the class of systems which are disjoint from every ergodic system}

\author{Eli Glasner and Benjamin Weiss}

\address{Department of Mathematics\\
     Tel Aviv University\\
         Tel Aviv\\
         Israel}
\email{glasner@math.tau.ac.il}

\address {Institute of Mathematics\\
 Hebrew University of Jerusalem\\
Jerusalem\\
 Israel}
\email{weiss@math.huji.ac.il}


\setcounter{secnumdepth}{2}



\setcounter{section}{0}



%




\begin{date}
{May 1,  2024}
\end{date}

\maketitle


\setcounter{secnumdepth}{2}


\setcounter{section}{0}


\begin{abstract}
In this note we give a fairly direct proof of a recent theorem of G\'{o}rska,  Lema\'{n}czyk and de la Rue
which characterises the class of measure preserving transformations that are disjoint from every ergodic 
measure preserving transformation. Our proof works just as well for any countable acting group.
\end{abstract}

\section*{Introduction}

  A recent preprint  \cite{{GLdR}} contains the following result. 
Let $\Ecal$ denote the class of ergodic measure preserving transformations,  and let $\Ecal^\perp$
denote those transformations that are disjoint from all elements of $\Ecal$.
Then,  $\Xb= (X, \Bcal, \mu, T) \in \Ecal^\perp$ if and only if, in the
 space  $(\bar{X}, \bar{\mu})$  of
ergodic components of  $\Xb$,
for $\bar{\mu} \times \bar{\mu}$-a.e. pair
$(\bar{x}_1, \bar{x}_2)$
the ergodic transformations $(T_{\bar{x}_1}, T_{\bar{x}_2})$ are disjoint.

What is surprising about this result is that it shows that $\Ecal^\perp$ is a much richer class than was previously thought.
The fact that $\Ecal^\perp$ is not empty is easy to see because the identity transformation on a non atomic
measure space (like $[0,1]$ with Lebesgue measure $\lambda$) is an element of $\Ecal^\perp$,
as the following simple argument shows.

Suppose that $\Xb$ is ergodic and $\rho$, a measure on $[0,1] \times X$, is a joining of the identity on $[0,1]$
and $\Xb$. Decompose $\rho$ over $\la$
$$
\rho = \int \del_t \times \rho_t \, d \la(t),
$$
so that $\int \rho_t \, d\la(t) = \mu$ and each $\rho_t$ is $T$-invariant.
Now, since $\Xb$ is ergodic this means that for $\la$-a.e. $t$
we have $\rho_t = \mu$, which shows that $\rho$ is the product measure $\la \times \mu$.

\br

We were intrigued by this result and found a fairly direct proof that works
just as well for any countable group $G$. In the following note we present this proof. For a general group we don't know 
how to construct non-trivial examples in $\Ecal^\perp$ , however this can be done for an arbitrary amenable group as we will show. 
In this connection we also give an extension to amenable groups of an old theorem of Andres del Junco \cite{dJ-81} to the effect that for a fixed 
ergodic transformation $T$ the set of those $S \in \Ecal$ that are disjoint from $T$ are a dense  $G_\del$.

It is easy to give a relative version of the main result and we point out in the penultimate section how this is done. 
In the final section we give a simple proof of the fact that $\Ecal(G)^{\perp}$ is closed under direct products
(Theorem 30 in \cite{BGdR} ).

\br

\section{The set of disjoint pairs is $G_\del$}
We denote by MPT the Polish group of all invertible probability measure preserving transformations
$([0,1], \Bcal, \mu, T)$, with $\mu$ the Lebesgue measure on $[0,1]$.
The subset of MPT consisting of the ergodic elements will be denoted by $\Ecal = \Ecal(\Z)$.
It is well known that $\Ecal$ is a $G_\del$-set.

For a pair of elements $(S,T)$ in MPT a {\em joining} is a probability measure $\la$ on $[0,1] \times [0,1]$
that projects onto Lebesgue measure $\mu$ on both coordinates
and is invariant under $S \times T$. A joining $\la$ defines a Markov operator $M = M_\la : L^2([0,1], \mu)
\to L^2([0,1], \mu)$ by
$$
Mf(x) = \int_0^1 f(y) \, d \la_x(y),
$$
where  the measures $\la_x$ arise via the disintegration of $\la$ over the first  coordinate:
$$
\la = \int_0^1 \del_x \times \la_x \, d\mu.
$$
We then have
\begin{align*}
\langle g, Mf \rangle
& = \int\int g(x) f(y) \, d \la_x(y)\, d\mu(x) =
\int g(x) f(y) \, d\la(x,y)\\
& = \int g(Sx) f(Ty) \, d\la(x,y)  = \langle  Sg, MTf \rangle\\
&= \langle g, S^{-1} M T f \rangle,
 \end{align*}
so that $M = S^{-1}MT$ or $MT = SM$.
Conversely, it is clear that every Markov operator $M$ which satisfies this
equality determines a joining $\la$ of $S$ and $T$.

Let $\Mcal$ denote the space of all the Markov operators $M : L^2( \mu) \to L^2( \mu)$.
We equip $\Mcal$ with the weak operator topology.
Recall that the map $T \mapsto U_T$, the Koopman operator associated with $T \in$ MPT,
is a homeomorphism from the Polish group MPT onto its image,  a closed subgroup of the group
$\Ucal$ of unitary operators on $L^2(\mu)$, where the latter is equipped with
the strong operator topology.

%

\begin{lem}
Let $(S_n,T_n) \to (S,T)$ be a convergent sequence in {\rm MPT} $\times$ {\rm MPT},
and for each $n$, let $\la_n$ be a joining of the pair $(S_n,T_n)$ with $M_n = M_{\la_n}$.
Suppose further that
the sequence $\la_n$ converges to a measure $\la$ with $M = M_\la$. Then $\la$ is a joining of $(S,T)$;
i.e. $M U_T = U_S M$.
\end{lem}

\begin{proof}
Let $f, g \in L^2(\mu)$, then
\begin{gather*}
|\langle f, M_n U_{T_n} g \rangle - \langle f, MU_{T}g \rangle | \\
 \leq
|\langle f, M_n U_{T_n} g \rangle - \langle f, M_nU_{T}g \rangle | \cdot
|\langle f, M_n U_{T} g \rangle - \langle f, MU_{T}g \rangle | \\
 \leq
\|f\| \cdot \|M_n\| \cdot \| U_{T_n} g - U_T g\|
+ |\langle f, M_n U_{T} g \rangle - \langle f, MU_{T}g \rangle |.
\end{gather*}
For sufficiently large $n$ both terms in the last line are small and we conclude that $M_n U_{T_n} \to MU_T$
in the weak operator topology. A similar calculation shows that $M_n U_{S_n} \to MU_S$ and hence $M U_T = U_S M$.
\end{proof}

\begin{prop}\label{Gdelta}
In MPT $\times$ MPT the set $\Dcal$ of pairs $(S,T)$ such that $S$ is disjoint from $T$ is $G_\del$.
\end{prop}

\begin{proof}

In the weak$^*$-topology the set of joinings corresponding to MPT $\times$ MPT form a closed
(hence compact) subset of the weak$^*$-compact metric space of probability measures on $[0,1] \times [0,1]$.
Fix a compatible metric $d$ on this space.
For $(S,T)$ in MPT $\times$ MPT let $J_{(S,T)}$ denote the closed set of all joinings of this pair.
This is always non-empty as the product measure $\mu \times \mu%
$ is  there. For $n =1,2,\dots$ set
$$
C_n = \{(S,T) : \diam J_{(S,T)} \geq 1/n\}.
$$
We claim that this is a closed set. Indeed, suppose $C_n \ni (S_k,T_k) \to (S,T)$.
For each $k$ there is a pair $\la_k , \la'_k \in J_{(S_k,T_k)}$ with $d(\la_k , \la'_k) \geq 1/n$.
We can assume that $\la_k \to \la$ and $\la'_k \to \la'$. Then $\la, \la' \in J_{(S,T)}$ and
$d(\la, \la') \geq 1/n$, whence $(S,T) \in C_n$.
Now it follows that the $G_\del$ set
$$
\Dcal  = \bigcap_{n =1}^\infty ({\rm MPT} \times {\rm MPT}) \setminus C_n
$$
coincides with the set of pairs $(S,T)$ with $J_{(S,T)} =\{\mu \times \mu\}$, namely
the set of disjoint pairs.
\end{proof}

\section{A characterization of the class $\Ecal^\perp$}

%
Let $(X, \mathcal{B}, \mu, T)$ be an invertible  probability measure preserving system.
Let
$$
\mu = \int_{\bar{X}} \mu_{\bar{x}} \, d {\bar{\mu}}(\bar{x})
$$
be the ergodic decomposition of $T$.
Assuming that the ergodic decomposition measure $\bar\mu$ has no atoms,
we can take the space of parameters $\bar{X}$ to be the unit interval $[0,1]$ with Lebesgue measure,
and let $X = \bar{X} \times Z$, with $Z$ also equals $[0,1]$.
Then, $T_{\bar{x}} = T \rest_{\{\bar{x}\} \times Z}$, or more precisely $T(\bar{x},z) = (\bar{x},T_{\bar{x}} z)$, and the systems
$(Z, \mu_{\bar{x}}, T_{\bar{x}})$ are the ergodic components of $T$.

\begin{prop}\label{prel}
Let $(X, \mathcal{B}, \mu, T)$ be an invertible  probability measure preserving system with ergodic decomposition
as above.
\begin{enumerate}
\item
If $T$ is disjoint from every ergodic $S$ in MPT then its ergodic decomposition measure  has no atoms.
\item
The set
$$
D = \{(\bar{x}_1, \bar{x}_2) \in \bar{X} \times \bar{X} : T_{\bar{x}_1} \perp T_{\bar{x}_2}\}
$$
is measurable.
\end{enumerate}
\end{prop}

\begin{proof}
The claim (1) is clear.

(2) The map $\bar{x} \mapsto T_{\bar{x}}$ is measurable, its image
$\Tcal = \{T_{\bar{x}}\}_{\bar{x} \in \bar{X}}$ is analytic in MPT, and therefore so
is the product set $\Tcal \times \Tcal$.
By Proposition \ref{Gdelta} the set $\Dcal$ is a $G_\del$-set, hence measurable and
finally we conclude that so is the set
$D = \Dcal \cap (\Tcal \times \Tcal)$.
\end{proof}

\begin{thm}\label{main}
Let $(X, \mathcal{B}, \mu, T)$ be an invertible  probability measure preserving system.
\begin{enumerate}
\item
If $T $ is disjoint from every ergodic $S$ in MPT, and
$\mu = \int_{\bar{X}} \mu_{\bar{x}} \, d \bar\mu(\bar{x})$ is its ergodic decomposition as above, then
the set
$$
D = \{(\bar{x}_1, \bar{x}_2) \in \bar{X} \times \bar{X} : T_{\bar{x}_1} \perp T_{\bar{x}_2}\}
$$
has $\bar\mu \times \bar\mu$ measure $1$.
\item
Conversely, if the ergodic decomposition measure of $T$ has no atoms
and $(\bar\mu \times \bar\mu)(D) =1$, then $(X, \mathcal{B}, \mu, T)$ is disjoint from all
ergodic systems $(Y, \mathcal{C}, \eta, S)$.
\end{enumerate}
\end{thm}

\begin{proof}
(1)  We will show that $D$ has $\bar\mu \times \bar\mu$ measure $1$.
 Suppose to the contrary that
 the symmetric set $P = \{(\bar{x}_1, \bar{x}_2) : T_{\bar{x}_1} \not\perp T_{\bar{x}_2}\}$ has positive
 $\bar\mu \times \bar\mu$-measure.
 By Fubini's theorem there is a point $\bar{x}_0 \in \bar{X}$ such that the set
 $K : = P_{\bar{x}_0} = \{\bar{x} \in \bar{X} : (\bar{x}, \bar{x}_0) \in P\}$ has positive $\bar\mu$-measure.
 Denote $(Y, \Ccal,\eta, S) := (Z, T_{\bar{x}_0}, \mu_{\bar{x}_0})$.
 We then have $T_{\bar{x}} \not \perp S$ for every $\bar{x} \in K$.


 For each $(\bar{x}_1, \bar{x}_2)  \in P$ let $J_{(\bar{x}_1, \bar{x}_2)}$ be
 the collection of non-product ergodic joinings of $T_{\bar{x}_1}$ and $T_{\bar{x}_2}$.
 By a theorem of Kuratowski and Ryll-Nardzewskii \cite{K-R} (see also \cite[Theorem 5.1]{P-72})
  there is a Borel cross-section
 $j : K \to \bigcup_{P} J_{(\bar{x}, \bar{x}_0)}, \ j(\bar{x}) \in J_{(\bar{x}, \bar{x}_0)}$,
 and then $\int_K  j(\bar{x}) \, d\bar\mu(\bar{x})$ is a non-trivial
 joining of $T$ and $S$, contradicting our assumption that $T$ is disjoint from all the ergodic systems.
 This concludes the proof of claim (1).

 \br

 We will next prove claim (2). So, we assume that  $(\bar\mu \times \bar\mu)(D) =1$ and
 we want to show that then $(X, \mathcal{B}, \mu, T)$ is disjoint from all
ergodic systems $(Y, \mathcal{C}, \eta, S)$.
Fix an ergodic system $(Y, \Ccal,\eta, S)$.
Suppose to the contrary that $T \not \perp S$. If for $\bar\mu$-a.e. $\bar{x}$
$T_{\bar{x}} \perp S $ then clearly $T \perp S$.
Thus there is a set $C \subset \bar{X}$ of positive $\bar\mu$-measure
such that $T_{\bar{x}} \not \perp S$ for every $\bar{x} \in C$.


\br

We have $(\bar\mu \times \bar\mu)((C \times C) \cap D) = \bar\mu(C)^2$; i.e.
the set $(C \times C) \cap D$ has full measure in $C\times C$.
Using transfinite induction one can show that there exists a subset $A \subset C$ of cardinality $\aleph _1$ such that
$A \times A \subset D$,
 alternatively, use \cite{L-02} to get a Cantor set $A$ with
$A \times A \subset D$ (hence of cardinality $2^{\aleph_0}$).
We then have:
\begin{enumerate}
\item[(i)]
 $\bar{x} \in A \imp T_{\bar{x}} \not \perp S $, and
 \item[(ii)]
 $T_{\bar{x}_1}  \perp T_{\bar{x}_2}$  for distinct $\bar{x}_1, \bar{x}_2 \in A$,
\end{enumerate}

In the sequel we will need the following lemma:

\begin{lem}
Let $(Y, \Ccal,\eta, S)$ and $(Z_i, \mathcal{F}_i, \theta_i, R_i), \ i \in I$ for some index set $I$,
be ergodic systems such that $S \not\perp R_i$ for every $i \in I$, and such that $R_i \perp R_j$
for every distinct $i, j \in I$.
Then there are  functions $g_i \in L^2(Y, \eta)$ such that the family $\{g_i : i \in I\}$ is orthogonal.
\end{lem}

\begin{proof}
Fix for each $i \in I$ a  joining $\la^{(i)}$ of  $(Y, \Ccal,\eta, S)$ and $(Z_i, \mathcal{F}_i, \theta_i, R_i)$
which is not the independent joining.
We write $\la^{(i)} = \int \del_y \times \la_y^{(i)} \, d \eta(y)$ and choose a function $h_i \in L^2(Z_i,\theta_i)$
with zero integral
such that the function
$$
g_i(y) = \int_{Z_i} h_i(z) \, d \la_y^{(i)}(z)
$$
is not constant. We then have
\begin{gather*}
\int_Y g_i(y) \, d\eta(y) = \int _Y \int_{Z_i} h_i(z) \, d \la_y^{(i)}(z)  \, d\eta(y) =\\
 =  \int_{Z_i} h_i(z) \, d\theta_i(z) = 0.
\end{gather*}
 Normalizing if necessary, we can assume that $g_i$ has $L^2$-norm $1$,

Now given distinct $i,j \in I$, in order to compute the inner product $\langle g_i, g_j \rangle$ in $L^2(Y,\eta)$,
consider the three-fold joining
$$
\Lambda = \int_{Y \times Z_i \times Z_j} \del_y \times \la_y^{(i)} \times \la_y^{(j)} \  d \eta(y).
$$
Its projection to the product space $Z_i \times Z_j$ is a joining of $R_i$ and $R_j$, hence must be the independet joining,
as these two transformations are disjoint.
We then have
\begin{align*}
\int_Y g_i(y) \ol{g_j(y)} d\, \eta(y)   & =
\int_Y \left( \int_{Z_i} h_i(z) \, d \la_y^{(i)}(z) \right)  \cdot
\left( \int_{Z_j} \ol{h_j(z')}\, d \la_y^{(i)}(z') \right)\, d\eta(y)\\
& =
\int_{Y \times Z_i \times Z_j}  g_i(y) \ol{g_j(y)} \ d \Lambda(y,z,z') \\
&   = \int_Y  g_i(y) \, d\eta(y) \cdot   \int_Y \ol{g_j(z)} \, d\eta(y) =0.
\end{align*}
\end{proof}

Applying this lemma with $I = A$ and $R_i = T_{\bar{x}}$ we obtain an orthonormal non-countable set
in the separable Hilbert space $L^2(Y, \eta)$, which is impossible.

This finishes the proof of claim (2).

\end{proof}

\section{The case of a general countable group}

Let $G$ be a countable  group.
Let $\mathbb{A(G)}$ denote the Polish space of probability measure preserving  $G$-actions.
The group MPT acts on $\A(G)$ by conjugation (see \cite{K-10}).
We let $\Ecal(G)$ denote the collection of ergodic systems in $\A(G)$. It is well known and
easy to see that this is a $G_\del$ subset of $\A(G)$.


As already mentioned in the introduction the notions of joinings, disjointness, etc. which were
defined in the preceding section for elements of MPT, naturally extend to the context of actions.
E.g. for a pair of actions $(a,b)$ in $\A(G) \times \A(G)$ a {\em joining} is still a probability measure $\la$ on $[0,1] \times [0,1]$
that projects onto Lebesgue measure $\mu$ on both coordinates
and is invariant under $a(g) \times b(g)$ for every $g \in G$.
Again, a joining $\la$,
written as
$$
\la = \int \del_x \times \la_x,
$$
 defines a Markov operator $M = M_\la : L^2([0,1], \mu) \to L^2([0,1], \mu)$ by
$$
Mf(x) = \int_0^1 f(y) \, d \la_x(y),
$$
and we then have $Ma(g) = b(g)M$ for every $g \in G$.
Also, conversely, every Markov operator $M$ which satisfies these
equalities determines a joining $\la$ of $a$ and $b$.

Moreover, examining the proofs of  Propositions \ref{Gdelta}, \ref{prel} and Theorem \ref{main}, we see that
(almost verbatim) they make good sense if we replace $\Z$ by an arbitrary countable group $G$.
We restate these results as follows:

\begin{prop}\label{GdeltaA}
In $\A(G) \times \A(G)$ the set $\Dcal$ of pairs $(a,b)$ such that $a$ is disjoint from $b$ is $G_\del$.
\end{prop}

\begin{prop}\label{prelA}
Let $\Xb = (X, \mathcal{B}, \mu, \Tb)$, with $\Tb = \{T_g\}_{g \in G}$,
be an element of $\A(G)$ with ergodic decomposition $\mu = \int_{\bar{X}} \mu_{\bar{x}} \, d \bar\mu(\bar{x})$.
\begin{enumerate}
\item
If $\Xb$ is disjoint from every ergodic $a$ in $\Ecal(G)$ then its ergodic decomposition measure  has no atoms.
\item
The set
$$
D = \{(\bar{x}_1, \bar{x}_2) \in \bar{X} \times \bar{X} : {\Tb}_{\bar{x}_1} \perp {\Tb}_{\bar{x}_2}\}
$$
is measurable.
\end{enumerate}
\end{prop}

\begin{thm}\label{mainA}
Let $\Xb = (X, \mathcal{B}, \mu, \Tb)$, with $\Rb = \{R_g\}_{g \in G}$,
be an element of $\A(G)$.
\begin{enumerate}
\item
If $\Tb$  is is disjoint from every ergodic $\Sb$ in $\Ecal(G)$, and
$\mu = \int_{\bar{X}} \mu_{\bar{x}} \, d \bar \mu(\bar{x})$ is its ergodic decomposition, then
the set
$$
D = \{(\bar{x}_1, \bar{x}_2) \in \bar{X} \times \bar{X} : \Tb_{\bar{x}_1} \perp \Tb_{\bar{x}_2}\}
$$
has $\bar\mu \times \bar\mu$ measure $1$.
\item
Conversely, if $(\bar\mu \times \bar\mu)(D) =1$, then $\Xb$
 is disjoint from all
ergodic systems $\Yb = (Y, \mathcal{C}, \eta, \Sb) \in \A(G)$.
\end{enumerate}
\end{thm}

%
%

\br

\section{The relative case}

In this short section we observe that,
with the obvious suitable modifications,
all our arguments so far go through when we consider the following
relative setup. Fix an ergodic dynamical system $\Zb = (Z, \mathcal{Z}, \ka, \Rb) \in \A(G)$, with $\Rb = \{R_g\}_{g \in G}$.
We now assume that all our systems admit $\Zb$ as a factor and denote this collection
of actions by $\A_{\Zb}(G)$.
Observe that when $\pi : \Xb \to \Zb$
is a factor map and $\mu = \int_{\bar{X}} \mu_{\bar{x}} \, d \bar\mu(\bar{x})$ is the ergodic decomposition
of $\mu$, then for $\bar\mu$-a.e. $\bar{x}$ we have $\pi_*(\mu_{\bar x}) = \ka$, so that
$\Xb \in \A_{\Zb}(G)$ implies $\Tb_{\bar{x}} \in  \A_{\Zb}(G)$ for $\bar{\mu}$-a.e. $\bar{x}$.

Recall that elements $\Xb = (X, \mathcal{B}, \mu, \Tb)$ and $\Yb = (Y, \mathcal{C}, \eta, \Sb)$ of $\A_{\Zb}(G)$
are {\em relatively disjoint over their common factor $\Zb$} if the relatively independent joining
$\mu \underset{\ka} \times \eta$ is the unique joining of $\Xb$ and $\Yb$ over $\Zb$.
We denote this relation by $\Xb \underset{\Zb} \perp \Yb$; see
e.g. \cite[Chapter 6, section 6]{Gl} for more details.

\br

\begin{thm}\label{mainR}
Let $\Xb = (X, \mathcal{B}, \mu, \Tb)$, with $\Tb = \{T_g\}_{g \in G}$,
be an element of $\A_{\Zb}(G)$ such that the map $\pi : \Xb \to \Zb$ is
$\ka$-a.e. infinite to one.
\begin{enumerate}
\item
If $\Tb$  is is disjoint over $\Zb$ from every ergodic $\Sb$ in $\A_{\Zb}(G)$, and
$\mu = \int_{\bar{X}} \mu_{\bar{x}} \, d \bar \mu(\bar{x})$ is its ergodic decomposition, then
the set
$$
D = \{(\bar{x}_1, \bar{x}_2) \in \bar{X} \times \bar{X} : \Tb_{\bar{x}_1} \underset{\Zb} \perp \Tb_{\bar{x}_2}\}
$$
has $\bar\mu \times \bar\mu$ measure $1$.
\item
Conversely, if $(\bar\mu \times \bar\mu)(D) =1$, then $\Xb$ is disjoint over $\Zb$ from all
ergodic systems $\Yb = (Y, \mathcal{C}, \eta, \Sb) \in \A_{\Zb}(G)$.
\end{enumerate}
\end{thm}

\br

\section{$\Ecal(G)^\perp$ systems for an amenable group}

We now further assume that our countable group  $G$ is amenable. We will show that the space $\Ecal(G)^\perp$
does not consist of  identity maps only.

\begin{thm}\label{non-Id}
Any countable infinite amenable group $G$ admits a non-identity action which is disjoint from every ergodic $G$-action.
\end{thm}

\begin{proof}

  In Proposition \ref{GdeltaA} we saw that the set $\Dcal$ of pairs of actions in $\A(G) \times \A(G)$ that are disjoint
  is  $G_{\delta}$.
  It is clearly non-empty  since any zero entropy action is disjoint from any Bernoulli
  shift. The group MPT $\times$ MPT acts by conjugation and clearly preserves $\Dcal$.
   By the Ornstein-Weiss theorem \cite{OW-80}, the conjugacy class of any free
ergodic action is dense in $\A(G)$, and thus $\Dcal$ is a dense $G_{\delta}$ and we can apply
the Mycielski-Kuratowski  theorem \cite[Theorem 19.1]{K-91} to deduce the existence of a Cantor set
$K \subset \mathbb{A}$ such that $K \times K \setminus {\text{diagonal}} \subset \Dcal $. Now, by Theorem \ref{main},
any continuous probability measure $\nu$ on $K$
will yield an element $\Xb \in \Ecal(G)^\perp$.

More explicitly, given any continuous probability measure $\nu$ on $K$, we define a  $G$-action
$\hat{T}_g : K \times X \to K \times X, \ g \in G$, by $\hat{T}_g (\Tb,x) = (\Tb, T_gx)$,
where $\Tb = \{T_g\}_{g \in G} \in K \subset \mathbb{A}$,
and elements of $\mathbb{A}$ are $G$-actions on the probability space $(X, \mathcal{B}, \mu)$.
The resulting dynamical system $(K \times X, \nu \times \mu, \hat{T})$ satisfies
the conditions of Theorem \ref{main} and is therefore an element of $\Ecal(G)^\perp$.
\end{proof}

 For the integers, Andres del Junco proved a stronger version of
 Prop. \ref {GdeltaA}. He showed
  that the set of $T$ that are disjoint from any fixed ergodic $S$ forms a dense
  $G_\delta$ subset of MPT. The analogue
  of his result for amenable groups is almost contained in \cite{F-W}, and we take the occasion to give a complete proof.


\br

\begin{thm}
Let $G$ be a countable amenable group. Then, for any fixed free ergodic action
$\Xb =(X, \Bcal, \mu, \{T_g\}_{g \in G}) \in \mathbb{A}(G)$, the set $\Xb^\perp$ is a dense $G_\del$ -set.
\end{thm}

\begin{proof}
By Proposition \ref{GdeltaA} we know that the set
of pairs $\Dcal = \{(\Xb, \Yb) : \Xb \perp \Yb\}$  is a $G_\del$ subset of $\mathbb{A} \times \mathbb{A}$.
Again, by Ornstein-Weiss' theorem \cite{OW-80}, it suffices to show that for any free ergodic action $\Xb$ there is
at least one free ergodic action $\Yb$ disjoint from it.
We consider the following cases:

\br

{\bf Case 1: $\Xb$ has zero entropy.}

In this case any Bernoulli shift is disjoint from $\Xb$.
Note that, as $0$-entropy actions are generic in $\mathbb{A}$ it follows that
$\Xb^\perp$ also contains a dense $G_\del$ subset of $0$-entropy systems.

\br

{\bf Case 2: $\Xb$ has completely positive entropy.}

In this case any zero-entropy system is in $\Xb^\perp$.

\br

{\bf Case 3: $\Xb$ has positive entropy but admits a non-trivial Pinsker factor.}

Let $\hat{\Xb}$  be the Pinsker factor of $\Xb$.
Choose a zero-entropy system $\Yb = (Y, \nu, \{S_g\}_{g \in G})$ in $\hat{\Xb}^\perp$ (see the remark in Case 1).
Then, since the extension $\Xb \to \hat{\Xb}$ is relatively completely positive entropy, it follows from
\cite{GTW} that also $\Xb \perp \Yb$.
\end{proof}

\br

\begin{question}
Does the same hold for an arbitrary countable infinite group $G$ ? I.e.
is $\mathcal{E}(G)^\perp \not = \{\text{identity map}\}$ for any such $G$ ?
\end{question}

\br


Here is a simple example for $F_2$.
Take $X$ to be the two-sphere $S^2$ with Lebesgue measure $\mu$.
For $\alpha \in [0,1)$ let $A_\alpha, B_\alpha$ be the rotations by $\alpha$
around the $z$ and $x$ axis respectively. This defines an $F_2$ strictly ergodic action, and for a.e. pair
$\alpha, \beta$ the corresponding actions are disjoint .
In fact, the product system $(X \times X, F_2)$, where $F_2$ acts via  $\langle A_\alpha, B_\alpha \rangle$
on the first $X$ and via $ \langle A_\beta, B_\beta \rangle$ on the second, is minimal.
As these systems are isometric this implies strict ergodicity. Thus, $\mu \times \mu$ is the
unique joining of the two systems.

Applying Theorem \ref{mainA} we conclude that the diagonal $F_2$-action on $[0,1) \times S^2$
defined by:
$$
T_a (t, x) = (t, A_tx), \qquad T_b (t, x) = (t, B_tx),
$$
(where $F_2 = \langle a, b \rangle$)
is an element of $\Ecal(F_2)^\perp$.

\br

\section{A short proof of the product theorem}

\begin{thm}
For any countable group $G$, the class $\Ecal(G)^\perp$ is closed under products.
\end{thm}

\begin{proof}
Again, in order to make the proof more transparent, we handel the case $G = \Z$.
The proof of the general case is similar (where one should replace, e.g. $T_{\bar{x}}$ by $\Tb_{\bar{x}}$
for $\Tb = \{T_g\}_{g \in G} \in \A(G)$).

Suppose $(X, \Bcal, \mu,T)$ and $(Y,\Ccal, \nu, S)$
are in $\Ecal^\perp$ and let $(Z, \Dcal, \rho, R)$ be ergodic.
We need to show that $T \times S$ is disjoint from $R$.
For $\bar{\mu} \times \bar{\nu}$-a.e. pair $(\bar{x}, \bar{y})$,
$T_{\bar{x}} \perp S_{\bar{y}}$ and therefore their product is ergodic.
It follows that
$$
\mu \times \nu = \int_{\bar{X} \times \bar{Y}} T_{\bar{x}} \times S_{\bar{y}} \ d (\bar{\mu} \times \bar{\nu})(\bar{x}, \bar{y})
$$
is the ergodic decomposition of $T \times S$.
Thus, we have to show that for a.e. pair $(\bar{x}, \bar{y})$,  \ $T_{\bar{x}} \times S_{\bar{y}} \perp R$.

Consider, for a fixed $\bar{y}$,  joinings of the triples $(T_{\bar{x}}, S_{\bar{y}}, R)$.
For $\bar{\mu}$ a.e. $\bar{x}$, \ $T_{\bar{x}} \perp R$. Thus, in these joinings $T_{\bar{x}}$ is independent of $R$.
Now $T_{\bar{x}} \times R$ is ergodic and hence, for $\bar{\nu}$-a.e. $\bar{y}$, $T_{\bar{x}} \times R$
is independent of $S_{\bar{y}}$, hence for such $\bar{y}$ all three components are independent for
$\bar{\mu} \times \bar{\nu}$-a.e. pair $(\bar{x}, \bar{y})$.
\end{proof}

\end{document}